\documentclass{article}

\usepackage{cmbright}
\usepackage{amssymb,amsmath,amsthm}
\usepackage{mathrsfs}
\usepackage{epsfig}
\usepackage[usenames,dvipsnames]{color}

\def\A{\mathcal A}
\def\B{\mathscr B}
\def\C{\mathbb C}
\def\d{\mathrm{d}}
\def\D{\mathscr D}

\def\H{\mathcal H}

\def\L{\mathscr L}
\def\M{\mathsf M}
\def\N{\mathbb N}

\def\R{\mathbb R}
\def\S{\mathbb S}
\def\SU{\mathsf{SU}}
\def\T{\mathbb T}
\def\U{\mathsf U}

\def\Z{\mathbb Z}

\def\dom{\mathcal D}
\def\Ran{\mathop{\mathsf{Ran}}\nolimits}

\def\e{\mathop{\mathrm{e}}\nolimits}

\def\linf{\mathsf{L}^{\:\!\!\infty}}
\def\ltwo{\mathsf{L}^{\:\!\!2}}

\DeclareMathOperator*{\slim}{s\;\!-lim\;\!}

\newtheorem{Theorem}{Theorem}[section]
\newtheorem{Remark}[Theorem]{Remark}

\newtheorem{Corollary}[Theorem]{Corollary}

\newtheorem{Example}[Theorem]{Example}

\begin{document}


\title{Commutator criteria for strong mixing}

\author{R. Tiedra de Aldecoa\footnote{Supported by the Chilean
Fondecyt Grant 1130168 and by the Iniciativa Cientifica Milenio ICM RC120002
``Mathematical Physics'' from the Chilean Ministry of Economy.}}

\date{\small}
\maketitle \vspace{-1cm}

\begin{quote}
\emph{
\begin{itemize}
\item[] Facultad de Matem\'aticas, Pontificia Universidad Cat\'olica de Chile,\\
Av. Vicu\~na Mackenna 4860, Santiago, Chile
\item[] \emph{E-mail:} rtiedra@mat.puc.cl
\end{itemize}
}
\end{quote}


\begin{abstract}
We present new criteria, based on commutator methods, for the strong mixing property
of discrete flows $\{U^N\}_{N\in\Z}$ and continuous flows $\{\e^{-itH}\}_{t\in\R}$
induced by unitary operators $U$ and self-adjoint operators $H$ in a Hilbert space
$\H$. Our approach put into evidence a general definition for the topological degree
of the curves $N\mapsto U^N$ and $t\mapsto \e^{-itH}$ in the unitary group of $\H$.
Among other examples, our results apply to skew products of compact Lie groups, time
changes of horocycle flows and adjacency operators on graphs.
\end{abstract}


\smallskip


\section{Introduction}
\setcounter{equation}{0}

In a recent series of papers \cite{FRT13,Tie14,Tie13,Tie12}, it has been shown that
commutator methods can be used to study the absolutely continuous spectrum of various
classes of ergodic dynamical systems. In the present work, we pursue this approach by
showing that the property of strong mixing can also be established using commutator
methods. More specifically, we present new criteria in terms of commutators for the
strong mixing property of two abstract classes of dynamical systems\;\!: discrete
flows $\{U^N\}_{N\in\Z}$ induced by a unitary operator $U$ and continuous flows
$\{\e^{-itH}\}_{t\in\R}$ induced by a self-adjoint operator $H$.

In the unitary case, our set-up is the following. Assume that $U$ and $A$ are
respectively a unitary and a self-adjoint operator in a Hilbert space $\H$, with $U$
such that the commutator $[A,U]$ exists and is bounded in some precise sense. Then,
the operator $[A,U]U^{-1}$ is bounded and self-adjoint, and one can define for all
$N\in\N_{\ge1}$ the Birkhoff averages
$\frac1N\sum_{n=0}^{N-1}U^n\big([A,U]U^{-1}\big)U^{-n}$. The strong limit of these
averages, namely the operator
$$
D:=\slim_{N\to\infty}\frac1N\sum_{n=0}^{N-1}U^n\big([A,U]U^{-1}\big)U^{-n},
$$
is (if it exists) a very interesting object. On the first hand, we explain why the
operator $D$ can be interpreted as a topological degree of the curve $N\mapsto U^N$ in
the unitary group $\U(\H)$ of $\H$ (Remark \ref{Rem_unitary}(b)). On the second hand,
we prove the following result (see Theorem \ref{thm_unitary} for a precise
statement): If the topological degree $D$ exists and satisfies some weak
regularity assumption with respect to $A$, then $U$ is strongly mixing and has purely
continuous spectrum in the subspace $\ker(D)^\perp$. (In some cases, we can even prove
that $U$ has purely absolutely continuous spectrum in the subspace $\ker(D)^\perp$,
see Remark \ref{Rem_unitary}(a) and Corollary \ref{second_cor_unitary}). This result
on strong mixing extends to a completely general hilbertian setting some related
result in the framework of skew products of compact Lie groups (see
\cite[Rem.~3.6]{Tie14} and references therein). Moreover, since it applies in
particular to Koopman operators, it provides a surprisingly simple and general
criterion to determine when an invertible measure-preserving transformation is
strongly mixing and has purely continuous spectrum in some subspace. Analogues of all
these results are also obtained in the case of a continuous flow
$\{\e^{-itH}\}_{t\in\R}$ induced by a self-adjoint operator $H$ (see Section
\ref{Sec_self}).

As an application of our results, we prove the strong mixing property or obtain purely
absolutely continuous spectrum in appropriate subspaces for operators $U$ and $A$
satisfying the Weyl commutation relation (Example \ref{ex_can_unitary}), for skew
products of compact Lie groups (Example \ref{ex_skew}), for operators $H$ and $A$
satisfying the Weyl commutation relation (Example \ref{ex_can_self}), for operators
$H$ and $A$ satisfying the homogeneous commutation relation (Example \ref{ex_hom}),
for adjacency operators on graphs (Example \ref{ex_graphs}), and for time changes of
horocycle flows on compact surfaces of constant negative curvature (Example
\ref{ex_horocycle}). In all of these examples, we obtain new results, or re-obtain
known results, with very simple and short proofs. In the case of skew products of
compact Lie groups, our results extend previous results of K. Fr{\c{a}}czek, A.
Iwanik, M. Lema\'nzyk, D. Rudolph and the author \cite{Fra00,Fra00_2,ILR93,Tie14}. In
the case of horocycle flows, we obtain a very simple proof that all time changes of
horocycle flows of class $C^1$ on compact surfaces of constant negative curvature are
strongly mixing. This complements classical results on horocycle flows of A. G.
Kushnirenko \cite{Kus74} and B. Marcus \cite{Mar77}. The diversity of the examples
covered in our paper and the simplicity of our proofs suggest that our results could
be applied to various other models in ergodic theory and spectral theory.

Here is a brief description of the content of the paper. In Section
\ref{Sec_commutators}, we recall the needed facts on commutators of operators and
regularity classes associated to them. In Section \ref{Sec_unitary}, we prove our
result on the strong mixing property for discrete flows $\{U^N\}_{N\in\Z}$ (Theorem
\ref{thm_unitary}) and give an interpretation of it in terms of the topological degree
$D$ (Remark \ref{Rem_unitary}(b)). We also present the application to skew products of
compact Lie groups (Example \ref{ex_skew}). In Section \ref{Sec_self}, we prove our
result on the strong mixing property for continuous flows $\{\e^{-itH}\}_{t\in\R}$
(Theorem \ref{thm_self}) and present the applications to adjacency operators on graphs
(Example \ref{ex_graphs}) and time changes of horocycle flows (Example
\ref{ex_horocycle}).

To conclude, we note that it would be interesting to see if the commutator criteria of
this paper admit an analogue for $C^*$-dynamical systems. Namely, given a triple
$\{\A,G,\alpha\}$ with $\A$ a $C^*$-algebra, $G$ a locally compact group and $\alpha$
a strongly continuous representation of $G$ in the automorphism group of $\A$, can we
find reasonable conditions on $\alpha$ in terms of an auxiliary map on $\A$ (maybe a
derivation of $\A$ replacing the map $[A,\;\!\cdot\;\!]$ appearing in our set-up)
guaranteeing that $\alpha$ is strong mixing?\\

\noindent
{\bf Acknowledgements.}
The author thanks K. Gelfert and M. Lema\'nzyk for interesting discussions on strong
mixing. The author is also grateful for their respective hospitality at the Federal
University of Rio de Janeiro in April 2014 and at the Nicolaus Copernicus University
in May 2014.

\section{Commutators and regularity classes}\label{Sec_commutators}
\setcounter{equation}{0}

We recall in this section some basic facts on commutators of operators and regularity
classes associated to them. We refer the reader to \cite[Chap.~5-6]{ABG96} for more
details.

Let $\H$ be a Hilbert space with scalar product
$\langle\;\!\cdot\;\!,\;\!\cdot\;\!\rangle$ antilinear in the first argument, denote
by $\B(\H)$ the set of bounded linear operators on $\H$, and write $\|\;\!\cdot\;\!\|$
both for the norm on $\H$ and the norm on $\B(\H)$. Let $A$ be a self-adjoint operator
in $\H$ with domain $\dom(A)$, and take $S\in\B(\H)$. For any $k\in\N$, we say that
$S$ belongs to $C^k(A)$, with notation $S\in C^k(A)$, if the map
\begin{equation}\label{eq_group}
\R\ni t\mapsto\e^{-itA}S\e^{itA}\in\B(\H)
\end{equation}
is strongly of class $C^k$. In the case $k=1$, one has $S\in C^1(A)$ if and only if
the quadratic form
$$
\dom(A)\ni\varphi\mapsto\big\langle\varphi,iSA\;\!\varphi\big\rangle
-\big\langle A\;\!\varphi,iS\varphi\big\rangle\in\C
$$
is continuous for the topology induced by $\H$ on $\dom(A)$. We denote by $[iS,A]$ the
bounded operator associated with the continuous extension of this form, or
equivalently the strong derivative of the function \eqref{eq_group} at $t=0$.
Moreover, if we set $A_\varepsilon:=(i\varepsilon)^{-1}(\e^{i\varepsilon A}-1)$ for
$\varepsilon\in\R\setminus\{0\}$, then we have that 
$\slim_{\varepsilon\searrow0}[iS,A_\varepsilon]=[iS,A]$ (see
\cite[Lemma~6.2.3(a)]{ABG96}).

Now, if $H$ is a self-adjoint operator in $\H$ with domain $\dom(H)$ and spectrum
$\sigma(H)$, we say that $H$ is of class $C^k(A)$ if $(H-z)^{-1}\in C^k(A)$ for some 
$z\in\C\setminus\sigma(H)$. In particular, $H$ is of class $C^1(A)$ if and only if the
quadratic form
$$
\dom(A)\ni\varphi\mapsto
\big\langle\varphi,(H-z)^{-1}A\hspace{1pt}\varphi\big\rangle
-\big\langle A\hspace{1pt}\varphi,(H-z)^{-1}\varphi\big\rangle\in\C
$$
extends continuously to a bounded form defined by the operator
$[(H-z)^{-1},A]\in\B(\H)$. In such a case, the set $\dom(H)\cap\dom(A)$ is a core for
$H$ and the quadratic form
$$
\dom(H)\cap\dom(A)\ni\varphi\mapsto\big\langle H\varphi,A\hspace{1pt}\varphi\big\rangle
-\big\langle A\hspace{1pt}\varphi,H\varphi\big\rangle\in\C
$$
is continuous in the topology of $\dom(H)$ (see \cite[Thm.~6.2.10(b)]{ABG96}). This
form then extends uniquely to a continuous quadratic form on $\dom(H)$ which can be
identified with a symmetric continuous operator $[H,A]$ from $\dom(H)$ to the adjoint
space $\dom(H)^*$. In addition, the following relation holds in $\B(\H)$ (see
\cite[Thm.~6.2.10(b)]{ABG96}):
$$
\big[(H-z)^{-1},A\big]=-(H-z)^{-1}[H,A](H-z)^{-1}.
$$

\section{Strong mixing for discrete flows}\label{Sec_unitary}
\setcounter{equation}{0}

We present in this section new criteria for the strong mixing property of a discrete
flow $\{U^N\}_{N\in\Z}$ induced by a unitary operator $U$ in a Hilbert space $\H$. We
start with the main theorem of the section, then present some of its corollaries, and
finally give some examples of applications.

We use the notation $\S^1:=\{z\in\C\mid|z|=1\}$ for the unit circle in the complex
plane and the notation $C^\infty_{\rm c}(\R\setminus\{0\})$ for the set of $C^\infty$
functions with compact support in $\R\setminus\{0\}$.

\begin{Theorem}[Strong mixing for discrete flows]\label{thm_unitary}
Let $U$ and $A$ be respectively a unitary and a self-adjoint operator in a Hilbert
space $\H$, with $U\in C^1(A)$. Assume that the strong limit
$$
D:=\slim_{N\to\infty}\frac1N\sum_{n=0}^{N-1}U^n\big([A,U]U^{-1}\big)U^{-n}
$$
exists, and suppose that $D^{-1}\eta(D)\;\!\dom(A)\subset\dom(A)$ for each
$\eta\in C^\infty_{\rm c}(\R\setminus\{0\})$. Then,
\begin{enumerate}
\item[(a)] $\lim_{N\to\infty}\big\langle\varphi,U^N\psi\big\rangle=0$ for each
$\varphi\in\ker(D)^\perp$ and $\psi\in\H$,
\item[(b)] $U|_{\ker(D)^\perp}$ has purely continuous spectrum.
\end{enumerate}
\end{Theorem}

Before the proof, we note that $[A,U]U^{-1}$ is bounded and self-adjoint. Thus, all
the operators $\frac1N\sum_{n=0}^{N-1}U^n\big([A,U]U^{-1}\big)U^{-n}$ are bounded and
self-adjoint, and so is their strong limit $D$. It follows that the operators
$D^{-1}\eta(D)$ with $\eta\in C^\infty_{\rm c}(\R\setminus\{0\})$ are well-defined by
functional calculus. We also note that $DU^n=U^nD$ for each $n\in\Z$. So,
$\ker(D)^\perp$ is a reducing subspace for $U$, and $U|_{\ker(D)^\perp}$ is a
well-defined unitary operator (see \cite[Ex.~5.39(b)]{Wei80}). Finally, we note that
the result of point (b) share some similarities with the criterion for the absence of
eigenvalues provided by the Virial theorem for unitary operators (see
\cite[Cor.~2.4~\&~Rem.~2.8]{FRT13}).

\begin{proof}[Proof of Theorem \ref{thm_unitary}]
(a) Since $U\in C^1(A)$, one has $U^N\in C^1(A)$ for each $N\in\N_{\ge1}$, and
\begin{align*}
\big[A,U^N\big]
=\sum_{n=0}^{N-1}U^{N-1-n}[A,U]U^n
&=\left(\sum_{n=0}^{N-1}U^{N-1-n}\big([A,U]U^{-1}\big)U^{n+1-N}\right)U^N\\
&=\left(\sum_{n=0}^{N-1}U^n\big([A,U]U^{-1}\big)U^{-n}\right)U^N.
\end{align*}
Therefore, one obtains that
\begin{equation}\label{eq_com_unitary}
\big[A,U^N\big]=ND_NU^N
\quad\hbox{with}\quad
D_N:=\frac1N\sum_{n=0}^{N-1}U^n\big([A,U]U^{-1}\big)U^{-n}.
\end{equation}
Now, let
$$
\D:=\big\{\varphi\in\dom(A)\mid\varphi=\eta(D)\varphi
~\hbox{for some}~\eta\in C^\infty_{\rm c}(\R\setminus\{0\})\big\},
$$
and take $\varphi\equiv\eta(D)\varphi\in\D$ and $\psi\in\dom(A)$. Then, the assumption
$D^{-1}\eta(D)\;\!\dom(A)\subset\dom(A)$ and \eqref{eq_com_unitary} imply for each
$N\in\N_{\ge1}$ that
\begin{align}
&\big|\big\langle\varphi,U^N\psi\big\rangle\big|\nonumber\\
&=\big|\big\langle DD^{-1}\eta(D)\varphi,
U^N\psi\big\rangle\big|\nonumber\\
&\le\big|\big\langle(D_N-D)D^{-1}\eta(D)\varphi,
U^N\psi\big\rangle\big|
+\big|\big\langle D^{-1}\eta(D)\varphi,D_NU^N\psi\big\rangle\big|\nonumber\\
&\le\big\|(D_N-D)D^{-1}\eta(D)\varphi\big\|\cdot\|\psi\|
+\frac1N\;\!
\big|\big\langle D^{-1}\eta(D)\varphi,\big[A,U^N\big]\psi\big\rangle\big|\nonumber\\
&=\big\|(D_N-D)D^{-1}\eta(D)\varphi\big\|\cdot\|\psi\|
+\frac1N\;\!\big|\big\langle AD^{-1}\eta(D)\varphi,U^N\psi\big\rangle
-\big\langle D^{-1}\eta(D)\varphi,U^NA\;\!\psi\big\rangle\big|\nonumber\\
&\le\big\|(D_N-D)D^{-1}\eta(D)\varphi\big\|\cdot\|\psi\|
+\frac1N\Big(\big\|AD^{-1}\eta(D)\varphi\big\|\cdot\|\psi\|
+\big\|D^{-1}\eta(D)\varphi\big\|\cdot\|A\;\!\psi\|\Big).
\label{eq_bound_unitary}
\end{align}
Since $\slim_{N\to\infty}D_N=D$, one infers that 
$\lim_{N\to\infty}\big\langle\varphi,U^N\psi\big\rangle=0$ for each $\varphi\in\D$ and
$\psi\in\dom(A)$, and thus the claim follows by the density of $\D$ in $\ker(D)^\perp$
and the density of $\dom(A)$ in $\H$.

(b) Suppose by absurd there exist $z\in\S^1$ and
$\varphi\in\ker(D)^\perp\setminus\{0\}$ such that $U\varphi=z\varphi$. Then, point (a)
implies that
$$
0
=\lim_{N\to\infty}\big|\big\langle\varphi,U^N\varphi\big\rangle\big|
=\lim_{N\to\infty}|z|^N\cdot\|\varphi\|^2
=\|\varphi\|^2
\ne0,
$$
which is a contradiction.
\end{proof}

\begin{Remark}\label{Rem_unitary}
(a) If we make in Theorem \ref{thm_unitary} the extra assumption that
$$
\sum_{N\ge1}
\left\|\left(\frac1N\sum_{n=0}^{N-1}U^n\big([A,U]U^{-1}\big)U^{-n}
-D\right)\varphi\right\|^2<\infty
~~\hbox{for each $\varphi\in\dom(A)$,}
$$
then we infer from \eqref{eq_bound_unitary} that 
$\sum_{N\ge1}\big|\big\langle\varphi,U^N\varphi\big\rangle\big|^2<\infty$ for each
$\varphi\in\D$. So, each $\varphi\in\D$ belongs to the absolutely continuous subspace
$\H_{\rm ac}(U)$ of $U$ (see the proof of \cite[Lemma~5]{ILR93}), and thus
$U|_{\ker(D)^\perp}$ has purely absolutely continuous spectrum since $\D$ is dense in
$\ker(D)^\perp$ and $\H_{\rm ac}(U)$ is closed in $\H$.

(b) The operator $D$ can be interpreted as a topological degree of the curve
$N\mapsto U^N$ in the unitary group $\U(\H)$ of $\H$. Indeed, one has
$$
D
=\slim_{N\to\infty}\frac1N\sum_{n=0}^{N-1}U^n\big([A,U]U^{-1}\big)U^{-n}
=\slim_{N\to\infty}\frac1N\big[A,U^N\big]U^{-N}.
$$
Therefore, if one considers the map $[A,\;\!\cdot\;\!]$ as a derivation on the set 
$\{U^N\}_{N\in\Z}\subset\U(\H)$, then $D$ can be interpreted as a renormalised winding
number of the curve $N\mapsto U^N$ in $\U(\H)$ (the logaritmic derivative
$\frac{\d z}z$ in the usual definition of winding number is replaced by the
``logaritmic derivative'' $\big[A,U^N\big]U^{-N}$ associated to $[A,\;\!\cdot\;\!]$).
Accordingly, one can rephrase the result of Theorem \ref{thm_unitary}(a) in the
following more intuitive terms\;\!: If $U$ is regular enough with respect to the
auxiliary operator $A$, and if the topological degree $D$ exists, then the flow
$\{U^N\}_{N\in\Z}$ is strongly mixing in the subspace where $D$ is nonzero. This
result extends to a completely general hilbertian setting some related result in the
framework of skew products of compact Lie groups (see \cite[Rem.~3.6]{Tie14} and
references therein).
\end{Remark}

The fact that $DU^n=U^nD$ for each $n\in\Z$ does not imply that $D$ is a function of
$U$ (one has to ask more for this, see \cite[Chap.~VI.75]{AG93}). Still, it suggests
to consider the particular case where $D=g(U)$ for some function $g$. This is what we
do in the next corollary.

We write $\chi_S$ for the characteristic function for a set $S$ and we denote by
$C^{k+\gamma}(\S^1;\R)$, $k\in\N$, $\gamma\in[0,1)$, the set of functions in
$C^k(\S^1;\R)$ with $k$-th derivative H\"older continuous of order $\gamma$ (with the
usual identification of $\S^1$ with $[0,1]$, see
\cite[Sec.~3.1.1~\&~Def.~3.2.5]{Gra08}).

\begin{Corollary}\label{first_cor_unitary}
Let $U$ and $A$ be respectively a unitary and a self-adjoint operator in a Hilbert
space $\H$, with $U\in C^1(A)$. Assume that the strong limit
$$
D:=\slim_{N\to\infty}\frac1N\sum_{n=0}^{N-1}U^n\big([A,U]U^{-1}\big)U^{-n}
$$
exists and satisfies $D=g(U)$ for some $g\in C^{2+\gamma}(\S^1;\R)$ with
$\gamma\in(0,1)$, and set $\Theta:=\S^1\setminus g^{-1}(\{0\})$. Then,
\begin{enumerate}
\item[(a)] $\lim_{N\to\infty}\big\langle\varphi,U^N\psi\big\rangle=0$ for each
$\varphi\in\chi_\Theta(U)\hspace{1pt}\H$ and $\psi\in\H$,
\item[(b)] $U|_{\chi_\Theta(U)\H}$ has purely continuous spectrum.
\end{enumerate}
\end{Corollary}

\begin{proof}
We have
$$
\chi_\Theta(U)\hspace{1pt}\H
=\chi_{\R\setminus\{0\}}\big(g(U)\big)\hspace{1pt}\H
=\ker\big(g(U)\big)^\perp.
$$
Therefore, the claims will follow from Theorem \ref{thm_unitary} if we show that
${g(U)}^{-1}\eta\big(g(U)\big)\dom(A)\subset\dom(A)$ for each
$\eta\in C^\infty_{\rm c}(\R\setminus\{0\})$. So, take
$\eta\in C^\infty_{\rm c}(\R\setminus\{0\})$. Then, the map
$z\mapsto{g(z)}^{-1}\eta\big(g(z)\big)$ belongs to $C^{2+\gamma}(\S^1)$ by the usual
properties of H\"older continuous functions, and thus the operator
${g(U)}^{-1}\eta\big(g(U)\big)$ can be written as a Fourier series
${g(U)}^{-1}\eta\big(g(U)\big)=\sum_{n\in\Z}c_nU^n$ with coefficients $c_n\in\C$
satisfying $|c_n|\le{\rm Const.}\;\!(1+|n|)^{-(2+\gamma)}$ for all $n\in\Z$ (see
\cite[Cor.~3.2.10(b)]{Gra08}). Therefore, one obtains for each $\varphi\in\dom(A)$ the
inequalities
\begin{align*}
\big|\big\langle A\;\!\varphi,{g(U)}^{-1}\eta\big(g(U)\big)\varphi\big\rangle
-\big\langle\varphi,{g(U)}^{-1}\eta\big(g(U)\big)A\;\!\varphi\big\rangle\big|
&\le\sum_{n\in\Z}|c_n|\cdot\big|\big\langle\varphi,
\big[A,U^n\big]\varphi\big\rangle\big|\\
&\le{\rm Const.}\;\!\big\|[A,U]\big\|\cdot\|\varphi\|^2
\sum_{n\in\Z}(1+|n|)^{-(2+\gamma)}|n|\\
&\le{\rm Const.}\;\!\big\|[A,U]\big\|\cdot\|\varphi\|^2.
\end{align*}
This implies that ${g(U)}^{-1}\eta\big(g(U)\big)\in C^1(A)$, and thus that
${g(U)}^{-1}\eta\big(g(U)\big)\;\!\dom(A)\subset\dom(A)$.
\end{proof}

The next result is a direct consequence of Remark \ref{Rem_unitary}(a) and Corollary \ref{first_cor_unitary}.

\begin{Corollary}\label{second_cor_unitary}
Let $U$ and $A$ be respectively a unitary and a self-adjoint operator in a Hilbert
space $\H$, with $U\in C^1(A)$. Assume that $[A,U]U^{-1}=g(U)$ for some
$g\in C^{2+\gamma}(\S^1;\R)$ with $\gamma\in(0,1)$, and set
$\Theta:=\S^1\setminus g^{-1}(\{0\})$. Then, $U|_{\chi_\Theta(U)\H}$ has purely
absolutely continuous spectrum.
\end{Corollary}

We close the section with some examples of applications of the previous results.

\begin{Example}[Canonical commutation relation]\label{ex_can_unitary}
Assume that $g(U)=g_0\in\R\setminus\{0\}$ in Corollary \ref{second_cor_unitary}. Then,
we have $\chi_\Theta(U)\hspace{1pt}\H=\H$, and thus Corollary \ref{second_cor_unitary} implies
that $U$ has purely absolutely continuous spectrum. In fact, we have in this case the
canonical commutation relation $[A/g_0,U]=U$ which is equivalent to the Weyl
commutation relation
$$
\e^{itA/g_0}U^n\e^{-itA/g_0}=\e^{itn}U^n,\quad t\in\R,~n\in\Z.
$$
Thus, representation theory implies that $U$ has Haar spectrum with uniform
multiplicity (see \cite[Thm.~3]{AP72}). This occurs for instance when $A$ and $U$ are
respectively a number operator and a bilateral shift on a Hilbert space (the bilateral
shift can be for example the Koopman operator of a $K$-automorphism in the
orthocomplement of the constant functions, see \cite[Sec.~3.1]{FRT13} and
\cite[Chap.~3.3]{Par81}).
\end{Example}

\begin{Example}[Skew products of compact Lie groups]\label{ex_skew}
Let $X$ be a compact Lie group with normalised Haar measure $\mu_X$ and let
$F\equiv\{F_t\}_{t\in\R}$ be a $C^\infty$ measure-preserving flow on $(X,\mu_X)$. The
family of operators $V_t:\ltwo(X,\mu_X)\to\ltwo(X,\mu_X)$ given by
$V_t\;\!\varphi:=\varphi\circ F_t$ defines a strongly continuous one-parameter unitary
group with self-adjoint generator $H$ in $\ltwo(X,\mu_X)$ which is essentially
self-adjoint on $C^\infty(X)$ and given by
$$
H\varphi:=-i\;\!\L_Y\varphi,\quad\varphi\in C^\infty(X),
$$
with $Y$ the divergence-free vector field associated to $F$ and $\L_Y$ the
corresponding Lie derivative.

Let $G$ be a second compact Lie group with normalised Haar measure $\mu_G$ and neutral
element $e_G$. Then, each function $\phi\in C(X;G)$ induces a cocycle
$X\times\Z\ni(x,n)\mapsto\phi^{(n)}(x)\in G$ over $F_1$ given by
$$
\phi^{(n)}(x):=
\begin{cases}
\phi(x)(\phi\circ F_1)(x)\cdots(\phi\circ F_{n-1})(x) & \hbox{if }n\ge1\\
e_G & \hbox{if }n=0\\
\big(\phi^{(-n)}\circ F_n\big)(x)^{-1}
& \hbox{if }n\le-1.
\end{cases}
$$
The skew product $T_\phi$ associated to $\phi$ defined by
$$
T_\phi:(X\times G,\mu_X\otimes\mu_G)\to(X\times G,\mu_X\otimes\mu_G),\quad
(x,g)\mapsto\big(F_1(x),g\;\!\phi(x)\big),
$$
is an automorphism of $(X\times G,\mu_X\otimes\mu_G)$, and the Koopman operator
$$
U_\phi\;\!\psi:=\psi\circ T_\phi,\quad\psi\in\H:=\ltwo(X\times G,\mu_X\otimes\mu_G),
$$
is a unitary operator in $\H$.

Let $\widehat G$ be the set of all (equivalence classes of) finite-dimensional
irreducible unitary representations of $G$. Then, each $\pi\in\widehat G$ is a
$C^\infty$ group homomorphism from $G$ to the unitary group $\U(d_\pi)$ of degree
$d_\pi:=\dim(\pi)<\infty$, and the Peter-Weyl theorem implies that the set of all
matrix elements $\{\pi_{jk}\}_{j,k=1}^{d_\pi}$ of all representations
$\pi\in\widehat G$ forms an orthogonal basis of $\ltwo(G,\mu_G)$. Accordingly, one has
the orthogonal decomposition
\begin{equation}\label{eq_decompo}
\H=\bigoplus_{\pi\in\widehat G}\bigoplus_{j=1}^{d_\pi}\H^{(\pi)}_j
\quad\hbox{with}\quad
\H^{(\pi)}_j:=\left\{\sum_{k=1}^{d_\pi}\varphi_k\otimes\pi_{jk}\mid
\varphi_k\in\ltwo(X,\mu_X),~k=1,\ldots,d_\pi\right\},
\end{equation}
and a calculation shows that $U_\phi$ is reduced by the decomposition
\eqref{eq_decompo}, with restriction $U_{\pi,j}:=U_\phi|_{\H^{(\pi)}_j}$ given by
$$
U_{\pi,j}\sum_{k=1}^{d_\pi}\varphi_k\otimes\pi_{jk}
=\sum_{k,\ell=1}^{d_\pi}(V_1\varphi_k)(\pi_{\ell k}\circ\phi)\otimes\pi_{j\ell}\;\!,
\quad\varphi_k\in\ltwo(X,\mu_X).
$$
Furthermore, the operator $A$ given by
$$
A\;\!\sum_{k=1}^{d_\pi}\varphi_k\otimes\pi_{jk}
:=\sum_{k=1}^{d_\pi}H\varphi_k\otimes\pi_{jk},
\quad\varphi_k\in C^\infty(X),
$$
is essentially self-adjoint in $\H^{(\pi)}_j$, and calculations as in
\cite[Lemma~3.2(c)~\&~Thm.~3.5]{Tie14} show the following\;\!: If for each
$k,\ell\in\{1,\ldots,d_\pi\}$ the function $\pi_{k\ell}\circ\phi\in C(X)$ satisfies
$\L_Y(\pi_{k\ell}\circ\phi)\in\linf(X)$, then $U_{\pi,j}\in C^1(A)$ with
$M:=[A,U_{\pi,j}](U_{\pi,j})^{-1}$ the matrix-valued multiplication operator in
$\H^{(\pi)}_j$ with matrix elements
$$
M_{k\ell}:=-i\;\!\big\{\L_Y(\pi\circ\phi)\cdot(\pi^*\circ\phi)\big\}_{k\ell}\;\!.
$$
Furthermore, one has for each $N\in\N_{\ge1}$
$$
D_N:=\frac1N\sum_{n=0}^{N-1}
(U_{\pi,j})^n\big([A,U_{\pi,j}](U_{\pi,j})^{-1}\big)(U_{\pi,j})^{-n}
=\frac1N\sum_{n=0}^{N-1}
\big(\pi\circ\phi^{(n)}\big)(M\circ F_n)\big(\pi^*\circ\phi^{(n)}\big),
$$
and $D_N$ reduces to the Birkhoff sum $\frac1N\sum_{n=0}^{N-1}M\circ F_n$ if the
matrix-valued function $\pi\circ\phi$ is diagonal.

We present now three cases where the limit $\slim_{N\to\infty}D_N$ exists and Theorem
\ref{thm_unitary} can be applied, but we stress that many other cases can be treated
since the method is general.

\noindent\\
{\bf Case 1.} Suppose that $X\times G=\T^d\times\T^{d'}$ with $d,d'\ge1$ and let
$F\equiv\{F_t\}_{t\in\R}$ be the flow on $\T^d$ given by
$$
F_t(x):=x+ty~\hbox{(mod $\Z^d$)},\quad t\in\R,~x\in\T^d,~y\in\R^d,
$$
with $y_1,y_2,\ldots,y_d,1$ rationally independent. Then, each element
$\chi_q\in\widehat{\T^{d'}}$ is a character of $\T^{d'}$ given by
$\chi_q(z):=\e^{2\pi iq\cdot z}$ for some $q\in\Z^{d'}$, one has
$\H^{(\pi)}_j\equiv\H^{(\chi_q)}_1=\ltwo(\T^d,\mu_{\T^d})\otimes\{\chi_q\}$, $F_1$ is
uniquely ergodic, and $\L_Y=y\cdot\nabla_x$. Suppose also that
$\phi\in C(\T^d;\T^{d'})$ satisfies $\phi=\xi+\eta$, where
\begin{enumerate}
\item[(i)] $\xi:\T^d\to\T^{d'}$ is given by $\xi(x):=Bx$ (mod $\Z^{d'}$) for some
integer matrix $B\in\M_{d'\times d}(\Z)$,
\item[(ii)] $\eta\in C(\T^d;\T^{d'})$ is such that $\L_Y(q\cdot\eta)\in\linf(\T^d)$
for each $q\in\Z^{d'}$.
\end{enumerate}
Then, a direct calculation (see \cite[Sec.~4.1]{Tie14}) shows that the matrix-valued
function $D_N$ reduces to the scalar function
$$
D_N=2\pi y\cdot\big(B^{\sf T}q\big)
+\frac{2\pi}N\sum_{n=0}^{N-1}\L_Y(q\cdot\eta)\circ F_n.
$$
Therefore, Birkhoff's pointwise ergodic theorem and Lebesgue dominated convergence
theorem imply that
$$
D:=\slim_{N\to\infty}D_N
=2\pi y\cdot\big(B^{\sf T}q\big)+2\pi\int_{\T^d}\d\mu_{\T^d}\,\L_Y(q\cdot\eta)
=2\pi y\cdot\big(B^{\sf T}q\big).
$$
Since $y_1,y_2,\ldots,y_d,1$ are rationally independent, one has $\ker(D)=\{0\}$ if
$B^{\sf T}q\ne0$. It follows from Theorem \ref{thm_unitary}(a) that
$\lim_{N\to\infty}\big\langle\varphi,(U_{\chi_q,1})^N\psi\big\rangle=0$ for each
$\varphi,\psi\in\H^{(\chi_q)}_1$ if $B^{\sf T}q\ne0$, and thus $U_\phi$ is strongly
mixing in the subspace
$\bigoplus_{q\in\Z^{d'}\!,\,B^{\sf T}q\ne0}\H^{(\chi_q)}_1\subset\H$.

In the case $d=d'=1$, this implies that $U_\phi$ is strongly mixing in
$\bigoplus_{q\in\Z\setminus\{0\}}\H^{(\chi_q)}_1$ if $\phi(x)=Bx+\eta(x)$, with
$B\in\Z\setminus\{0\}$ and $\eta:\T\to\T$ absolutely continuous with
$\eta'\in\linf(\T)$. A. Iwanik, M. Lema\'nzyk and D. Rudolph proved in
\cite[Sec.~3]{ILR93} the same result without the assumption $\eta'\in\linf(\T)$ by
studying directly the Fourier coefficients of the operators $U_{\chi_q,1}$. In the
case $d>1$ or $d'>1$, the result seems to be new, but we refer to the paper of K.
Fr{\c{a}}czek \cite{Fra00} for closely related results in the case of extensions of
rotations on $\T^d$ by cocycles $\Z^d\times\T^d\to\T$.

\noindent\\
{\bf Case 2.} Suppose that $X\times G=\T^d\times\SU(2)$ with $d\ge1$ and let
$F\equiv\{F_t\}_{t\in\R}$ be the flow on $\T^d$ given by
$$
F_t(x):=x+ty~\hbox{(mod $\Z^d$)},\quad t\in\R,~x\in\T^d,~y\in\R^d,
$$
with $y_1,y_2,\ldots,y_d,1$ rationally independent. Then, each $(n+1)$-dimensional
representation $\pi^{(n)}\in\widehat{\SU(2)}$ can be described in terms of homogeneous
polynomials of degree $n$ in two variables (see \cite[Chap.~II]{Sug90}), $F_1$ is
uniquely ergodic, and $\L_Y=y\cdot\nabla_x$. Suppose also that
$\phi\in C\big(\T^d;\SU(2)\big)$ satisfies
$$
\phi(x):=h
\begin{pmatrix}
\e^{2\pi i(b\cdot x+\eta(x))} & 0\\
0 & \e^{-2\pi i(b\cdot x+\eta(x))}
\end{pmatrix}
h^*,
\quad x\in\T^d,
$$
with $h\in\SU(2)$, $b\in\Z^d\setminus\{0\}$ and $\eta\in C(\T^d;\R)$ such that 
$\L_Y(\eta)\in\linf(\T^d)$. Then, a calculation using the expression of the matrix
elements $\pi^{(n)}_{jk}$ (see \cite[Sec.~4.2]{Tie14}) shows that
$$
D_N=2\pi\left((y\cdot b)+\frac1N\sum_{m=0}^{N-1}\L_Y(\eta)\circ F_m\right)
\begin{pmatrix}
0!(n-0)!(2\cdot0-n) && 0\\
& \ddots &\\
0 && n!(n-n)!(2\cdot n-n)
\end{pmatrix}.
$$
Therefore, Birkhoff's pointwise ergodic theorem and Lebesgue dominated convergence
theorem imply that
\begin{align*}
D:=\slim_{N\to\infty}D_N
&=2\pi\left((y\cdot b)+\int_{\T^d}\d\mu_{\T^d}\,\L_Y(\eta)\right)
\begin{pmatrix}
0!(n-0)!(2\cdot0-n) && 0\\
& \ddots &\\
0 && n!(n-n)!(2\cdot n-n)
\end{pmatrix}\\
&=2\pi(y\cdot b)
\begin{pmatrix}
0!(n-0)!(2\cdot0-n) && 0\\
& \ddots &\\
0 && n!(n-n)!(2\cdot n-n)
\end{pmatrix}.
\end{align*}
Since $y_1,y_2,\ldots,y_d,1$ are rationally independent, one has $(y\cdot b)\ne0$, and
thus
$$
\ker(D)
=\ker(E)
\quad\hbox{with}\quad
E:=\begin{pmatrix}
(2\cdot0-n) && 0\\
& \ddots &\\
0 && (2\cdot n-n)
\end{pmatrix}.
$$
Since $E$ is invertible when $n\in2\N+1$, it follows from Theorem \ref{thm_unitary}(a)
that $\lim_{N\to\infty}\big\langle\varphi,(U_{\pi^{(n)},j})^N\psi\big\rangle=0$ for
each $\varphi,\psi\in\H^{(\pi^{(n)})}_j$ when $n\in2\N+1$. Thus, $U_\phi$ is strongly
mixing in the subspace
$\bigoplus_{n\in2\N+1}\bigoplus_{j=0}^n\H^{(\pi^{(n)})}_j\subset\H$.

This result on strong mixing complements previous works of K. Fr{\c{a}}czek and of the
author, where it is proved that $U_\phi$ has countable spectrum in the subspace $\bigoplus_{n\in2\N+1}\bigoplus_{j=0}^n\H^{(\pi^{(n)})}_j$ if the function
$\phi:\T^d\to\SU(2)$ satisfies additional regularity assumptions (see
\cite[Thm.~6.1~\&~Thm.~8.2]{Fra00_2} when $d=1,2$ and \cite[Thm.~4.5]{Tie14} when
$d\ge1$ is arbitrary).

\noindent\\
{\bf Case 3.} Suppose that $X\times G=\T^d\times\U(2)$ with $d\ge1$ and let
$F\equiv\{F_t\}_{t\in\R}$ be the flow on $\T^d$ given by
$$
F_t(x):=x+ty~\hbox{(mod $\Z^d$)},\quad t\in\R,~x\in\T^d,~y\in\R^d,
$$
with $y_1,y_2,\ldots,y_d,1$ rationally independent. Then, the set $\widehat{\U(2)}$
coincides (up to unitary equivalence) with the set of tensors products
$\{\rho_{2m-n}\otimes\pi^{(n)}\}_{m\in\Z,n\in\N}$, with $\pi^{(n)}\in\widehat{\SU(2)}$
and $\rho_{2m-n}:\S^1\to\S^1$ given by $\rho_{2m-n}(z):=z^{2m-n}$ (see
\cite[Sec.~II.5]{BtD85}), $F_1$ is uniquely ergodic, and $\L_Y=y\cdot\nabla_x$.
Suppose also that
$$
\phi(x)=h
\begin{pmatrix}
\e^{2\pi i(b_1\cdot x+\eta_1(x))} & 0\\
0 & \e^{2\pi i(b_2\cdot x+\eta_2(x))}
\end{pmatrix}
h^*,\quad x\in\T^d,
$$
with $h\in\U(2)$, $b_1,b_2\in\Z^d$ and $\eta_1,\eta_2\in C(\T^d;\R)$ such that
$\L_Y(\eta_1),\L_Y(\eta_2)\in\linf(\T^d)$. Then, using results of
\cite[Sec.~4.3]{Tie14} one can prove the existence of the limit
$D:=\slim_{N\to\infty}D_N$ and apply Theorem \ref{thm_unitary}(a) for some indices
$(m,n)\in\Z\times\N$ and some values of $b_1,b_2\in\Z^d$. We leave this time the
details to the reader and just give the final result\hspace{1pt}: Write
$b_\pm:=b_1\pm b_2$ and set
$$
R:=\left\{(m,n)\in\Z\times\N\mid
\inf_{k\in\{0,\ldots,n\}}\big|(2m-n)(b_+\cdot y)+(2k-n)(b_-\cdot y)\big|>0\right\}.
$$
Then, $U_\phi$ is strongly mixing in the subspace
$\bigoplus_{(m,n)\in R}\bigoplus_{j=0}^n\H^{(\rho_{2m-n}\otimes\pi^{(n)})}_j\subset\H$.

This result on strong mixing complements \cite[Thm.~4.9]{Tie14}, where it is proved
that $U_\phi$ has countable Lebesgue spectrum in the subspace
$\bigoplus_{(m,n)\in R}\bigoplus_{j=0}^n\H^{(\rho_{2m-n}\otimes\pi^{(n)})}_j$ if
$\L_Y(\eta_1)$ and $\L_Y(\eta_2)$ satisfy an additional regularity assumption of
Dini-type.
\end{Example}

\section{Strong mixing for continuous flows}\label{Sec_self}
\setcounter{equation}{0}

We present in this section new criteria for the strong mixing property of a continuous
flow $\{\e^{-itH}\}_{t\in\R}$ induced by a self-adjoint operator $H$ in a Hilbert
space $\H$. We start with the main theorem of the section, then present some of its
corollaries, and finally give some examples of applications.

\begin{Theorem}[Strong mixing for continuous flows]\label{thm_self}
Let $H$ and $A$ be self-adjoint operators in a Hilbert space $\H$, with
$(H-i)^{-1}\in C^1(A)$. Assume that the strong limit
$$
D:=\slim_{t\to\infty}\frac1t\int_0^t\d s\,\e^{isH}(H+i)^{-1}[iH,A](H-i)^{-1}\e^{-isH}
$$
exists, and suppose that $D^{-1}\eta(D)\;\!\dom(A)\subset\dom(A)$ for each
$\eta\in C^\infty_{\rm c}(\R\setminus\{0\})$. Then,
\begin{enumerate}
\item[(a)] $\lim_{t\to\infty}\big\langle\varphi,\e^{-itH}\psi\big\rangle=0$ for each
$\varphi\in\H$ and $\psi\in\ker(D)^\perp$,
\item[(b)] $H|_{\ker(D)^\perp}$ has purely continuous spectrum.
\end{enumerate}
\end{Theorem}

Before the proof, we note that $(H+i)^{-1}[iH,A](H-i)^{-1}$ is bounded and
self-adjoint. Thus, all the operators
$\frac1t\int_0^t\d s\,\e^{isH}(H+i)^{-1}[iH,A](H-i)^{-1}\e^{-isH}$ are bounded and
self-adjoint, and so is their strong limit $D$. It follows that the operators
$D^{-1}\eta(D)$ with $\eta\in C^\infty_{\rm c}(\R\setminus\{0\})$ are well-defined by
functional calculus. We also note that $D\e^{itH}=\e^{itH}D$ for each $t\in\R$. So,
$D\;\!\chi_B(H)=\chi_B(H)D$ for each Borel set $B\subset\R$. Thus, $\ker(D)^\perp$ is
a reducing subspace for $H$, and $H|_{\ker(D)^\perp}$ is a well-defined self-adjoint
operator (see \cite[Thm.~7.28]{Wei80}). Finally, we note that the result of point (b)
share some similarities with the criterion for the absence of eigenvalues provided by
the Virial theorem for self-adjoint operators (see \cite[Cor.~7.2.11]{ABG96}).

\begin{proof}[Proof of Theorem \ref{thm_self}]
(a) Since $(H-i)^{-1}\in C^1(A)$, we have $(H-i)^{-1}\;\!\dom(A)\subset\dom(A)$,
and the operator
$$
\widetilde A\;\!\varphi:=(H+i)^{-1}A\;\!(H-i)^{-1}\varphi,\quad\varphi\in\dom(A),
$$
is essentially self-adjoint (see \cite[Lemma~7.2.15]{ABG96}). Take
$\varphi\in\dom(A)$ and $t>0$, and define for $\varepsilon\in\R\setminus\{0\}$ the
operator $A_\varepsilon:=(i\varepsilon)^{-1}(\e^{i\varepsilon A}-1)$. Then, we have
the equalities
\begin{align}
&\big\langle\widetilde A\;\!\varphi,\e^{-itH}\varphi\big\rangle
-\big\langle\varphi,\e^{-itH}\widetilde A\;\!\varphi\big\rangle\nonumber\\
&=\lim_{\varepsilon\searrow0}\Big(\big\langle\varphi,
(H+i)^{-1}A_\varepsilon\;\!(H-i)^{-1}\e^{-itH}\varphi\big\rangle
-\big\langle\varphi,
\e^{-itH}(H+i)^{-1}A_\varepsilon\;\!(H-i)^{-1}\varphi\big\rangle\Big)\nonumber\\
&=\lim_{\varepsilon\searrow0}\int_0^t\d s\,\frac\d{\d s}\;\!\big\langle\varphi,
\e^{i(s-t)H}(H+i)^{-1}A_\varepsilon\;\!(H-i)^{-1}\e^{-isH}\varphi\big\rangle\nonumber\\
&=\lim_{\varepsilon\searrow0}\int_0^t\d s\,\big\langle\varphi,
\e^{i(s-t)H}(H+i)^{-1}[iH,A_\varepsilon](H-i)^{-1}\e^{-isH}\varphi\big\rangle.
\label{eq_lim_int}
\end{align}
Now, since $(H-i)^{-1}\in C^1(H)$, we have
$$
\slim_{\varepsilon\searrow0}(H+i)^{-1}[iH,A_\varepsilon](H-i)^{-1}
=(H+i)^{-1}[iH,A](H-i)^{-1}.
$$
Therefore, we can interchange the limit and the integral in \eqref{eq_lim_int} and
obtain
\begin{align*}
\big\langle\widetilde A\;\!\varphi,\e^{-itH}\varphi\big\rangle
-\big\langle\varphi,\e^{-itH}\widetilde A\;\!\varphi\big\rangle
&=\int_0^t\d s\,\big\langle\varphi,
\e^{i(s-t)H}(H+i)^{-1}[iH,A](H-i)^{-1}\e^{-isH}\varphi\big\rangle\\
&=\big\langle\varphi,t\e^{-itH}D_t\;\!\varphi\big\rangle
\end{align*}
with $D_t:=\frac1t\int_0^t\d s\,\e^{isH}(H+i)^{-1}[iH,A](H-i)^{-1}\e^{-isH}$. Since
$\dom(A)$ is a core for $\widetilde A$, this implies that
$\e^{-itH}\in C^1(\widetilde A)$ with
\begin{equation}\label{eq_com_self}
\big[\widetilde A,\e^{-itH}\big]=t\e^{-itH}D_t.
\end{equation}
Now, let
$$
\D:=\big\{\psi\in\dom(A)\mid\psi=\eta(D)\psi
~\hbox{for some}~\eta\in C^\infty_{\rm c}(\R\setminus\{0\})\big\},
$$
and take $\varphi\in\dom(A)$ and $\psi\equiv\eta(D)\psi\in\D$. Then, the assumption
$D^{-1}\eta(D)\;\!\dom(A)\subset\dom(A)$ and \eqref{eq_com_self} imply for each $t>0$
that
\begin{align}
&\big|\big\langle\varphi,\e^{-itH}\psi\big\rangle\big|\nonumber\\
&=\big|\big\langle\varphi,\e^{-itH}DD^{-1}\eta(D)\psi\big\rangle\big|\nonumber\\
&\le\big|\big\langle\varphi,\e^{-itH}(D_t-D)D^{-1}\eta(D)\psi\big\rangle\big|
+\big|\big\langle\varphi,\e^{-itH}D_t\;\!D^{-1}\eta(D)\psi\big\rangle\big|\nonumber\\
&\le\|\varphi\|\cdot\big\|(D_t-D)D^{-1}\eta(D)\psi\big\|
+\frac1t\;\!\big|\big\langle\varphi,
\big[\widetilde A,\e^{-itH}\big]D^{-1}\eta(D)\psi\big\rangle\big|\nonumber\\
&=\|\varphi\|\cdot\big\|(D_t-D)D^{-1}\eta(D)\psi\big\|
+\frac1t\;\!\big|\big\langle\widetilde A\;\!\varphi,\e^{-itH}D^{-1}\eta(D)\psi\big\rangle
-\big\langle\varphi,\e^{-itH}\widetilde A\;\!D^{-1}\eta(D)\psi\big\rangle\big|\nonumber\\
&\le\|\varphi\|\cdot\big\|(D_t-D)D^{-1}\eta(D)\psi\big\|
+\frac1t\Big(\big\|\widetilde A\;\!\varphi\big\|\cdot\big\|D^{-1}\eta(D)\psi\big\|
+\|\varphi\|\cdot\big\|\widetilde A\;\!D^{-1}\eta(D)\psi\big\|\Big).
\label{eq_bound_self}
\end{align}
Since $\slim_{t\to\infty}D_t=D$, one infers that
$\lim_{t\to\infty}\big\langle\varphi,\e^{-itH}\psi\big\rangle=0$ for each
$\varphi\in\dom(A)$ and $\psi\in\D$, and thus the claim follows by the density of
$\dom(A)$ in $\H$ and the density of $\D$ in $\ker(D)^\perp$.

(b) Suppose by absurd there exist $\lambda\in\R$ and
$\psi\in\ker(D)^\perp\setminus\{0\}$ such that $H\psi=\lambda\psi$. Then, point (a)
implies that
$$
0
=\lim_{t\to\infty}\big|\big\langle\psi,\e^{-itH}\psi\big\rangle\big|
=\lim_{t\to\infty}\big|\e^{-it\lambda}\big|\cdot\|\psi\|^2
=\|\psi\|^2
\ne0,
$$
which is a contradiction.
\end{proof}

\begin{Remark}\label{Rem_self}
(a) If we make in Theorem \ref{thm_self} the extra assumption that
$$
\left\|\left(\frac1t\int_0^t\d s\,\e^{isH}(H+i)^{-1}[iH,A](H-i)^{-1}\e^{-isH}
-D\right)\psi\right\|\in\ltwo\big([1,\infty),\d t\big)
~~\hbox{for each $\psi\in\dom(A)$,}
$$
then we infer from \eqref{eq_bound_self} that
$\big|\big\langle\psi,\e^{-itH}\psi\big\rangle\big|\in\ltwo\big([1,\infty),\d t\big)$
for each $\psi\in\D$. So, each $\psi\in\D$ belongs to the absolutely continuous
subspace $\H_{\rm ac}(H)$ of $H$ (see the proof of \cite[Thm.~XIII.23]{RS79}), and
thus $H|_{\ker(D)^\perp}$ has purely absolutely continuous spectrum since $\D$ is
dense in $\ker(D)^\perp$ and $\H_{\rm ac}(H)$ is closed in $\H$.

(b) A Cayley transform shows that considering Birkhoff averages of the operator
$(H+i)^{-1}[iH,A](H-i)^{-1}$ in the self-adjoint case is consistent with considering
Birkhoff averages of the operator $[A,U]U^{-1}$ in the unitary case. Indeed, if $U$
and $A$ are as in Theorem \ref{thm_unitary} and $1$ does not belong to the point
spectrum of $U$, then the range $\Ran(1-U)$ of $1-U$ is dense in $\H$, and the Cayley
transform of $U$ given by
$$
H\varphi:=i\;\!(1+U)(1-U)^{-1}\varphi,
\quad\varphi\in\dom(H):=\Ran(1-U),
$$
is a self-adjoint operator \cite[Thm.~8.4(b)]{Wei80}. Moreover, one has the identities
$$
U^{-1}=(H+i)(H-i)^{-1}
\quad\hbox{and}\quad
(H+i)^{-1}=\frac1{2i}\;\!(1-U),
$$
and thus $H$ is of class $C^1(A)$ with
\begin{align*}
(H+i)^{-1}[iH,A](H-i)^{-1}
=i\big[A,(H+i)^{-1}\big](H+i)(H-i)^{-1}
=-\frac12\;\![A,U]U^{-1}.
\end{align*}
\end{Remark}

In analogy to the unitary case, we consider the particular case where $D=g(H)$ for
some function $g$.

\begin{Corollary}\label{first_cor_self}
Let $H$ and $A$ be self-adjoint operators in a Hilbert space $\H$, with
$(H-i)^{-1}\in C^1(A)$. Assume that the strong limit
$$
D:=\slim_{t\to\infty}\frac1t\int_0^t\d s\,\e^{isH}(H+i)^{-1}[iH,A](H-i)^{-1}\e^{-isH}
$$
exists and satisfies $D=g(H)$ for some $g\in C^2(\R;\R)\cap\linf(\R)$, and set
$\Lambda:=\R\setminus g^{-1}(\{0\})$. Then,
\begin{enumerate}
\item[(a)] $\lim_{t\to\infty}\big\langle\varphi,\e^{-itH}\psi\big\rangle=0$ for each
$\varphi\in\H$ and $\psi\in\chi_\Lambda(H)\hspace{1pt}\H$,
\item[(b)] $H|_{\chi_\Lambda(H)\H}$ has purely continuous spectrum.
\end{enumerate}
\end{Corollary}

\begin{proof}
We have
$$
\chi_\Lambda(H)\hspace{1pt}\H
=\chi_{\R\setminus\{0\}}\big(g(H)\big)\hspace{1pt}\H
=\ker\big(g(H)\big)^\perp,
$$
and we have ${g(H)}^{-1}\eta\big(g(H)\big)\dom(A)\subset\dom(A)$ for each
$\eta\in C^\infty_{\rm c}(\R\setminus\{0\})$ due to \cite[Thm.~6.2.5]{ABG96}\footnote{As
in Corollary \ref{first_cor_unitary}, one can give a more direct proof of the
inclusion ${g(H)}^{-1}\eta\big(g(H)\big)\dom(A)\subset\dom(A)$ using Fourier series if
one assumes that $g\in C^{2+\gamma}(\R;\R)\cap\linf(\R)$ with $\gamma\in(0,1)$.}.
So, the claims follow from Theorem \ref{thm_self}.
\end{proof}

The next result is a direct consequence of Remark \ref{Rem_self}(a) and Corollary
\ref{first_cor_self}.

\begin{Corollary}\label{second_cor_self}
Let $H$ and $A$ be self-adjoint operators in a Hilbert space $\H$, with
$(H-i)^{-1}\in C^1(A)$. Assume that $(H+i)^{-1}[iH,A](H-i)^{-1}=g(H)$ for some
$g\in C^2(\R;\R)\cap\linf(\R)$, and set $\Lambda:=\R\setminus g^{-1}(\{0\})$. Then,
$H|_{\chi_\Lambda(H)\H}$ has purely absolutely continuous spectrum.
\end{Corollary}

We close the section with some examples of applications of the previous results.

\begin{Example}[Canonical commutation relation]\label{ex_can_self}
Assume that $g(H)=(H+i)^{-1}g_0(H-i)^{-1}$ with $g_0\in\R\setminus\{0\}$ in Corollary
\ref{second_cor_self}. Then, we have $\chi_\Lambda(H)\hspace{1pt}\H=\H$, and thus Corollary
\ref{second_cor_self} implies that $H$ has purely absolutely continuous spectrum. In
fact, we have in this case the canonical commutation relation $[iH,A/g_0]=1$ which is
equivalent to the Weyl commutation relation
$$
\e^{-itA/g_0}\e^{isH}\e^{itA/g_0}=\e^{ist}\e^{isH},\quad s,t\in\R.
$$
Thus, Stone-von Neumann theorem implies that $H$ has Lebesgue spectrum with uniform
multiplicity.
\end{Example}

\begin{Example}[Homogeneous commutation relation]\label{ex_hom}
Assume that $g(H)=(H+i)^{-1}g_0H(H-i)^{-1}$ with $g_0\in\R\setminus\{0\}$ in Corollary
\ref{second_cor_self}. Then, we have $\chi_\Lambda(H)\hspace{1pt}\H=\ker(H)^\perp$, and thus
Corollary \ref{second_cor_self} implies that $H|_{\ker(H)^\perp}$ has purely
absolutely continuous spectrum. In fact, we have in this case the homogeneous
commutation relation $[iH,A/g_0]=H$ which is equivalent to the commutation relation
$$
\e^{-itA/g_0}\e^{isH}\e^{itA/g_0}=\e^{ise^tH},\quad s,t\in\R,
$$
and it is known that this implies that $H|_{\ker(H)^\perp}$ has Lebesgue spectrum with
uniform multiplicity if there exists a unitary operator $W$ in $\H$ such that
$W\e^{-itA/g_0}W^{-1}=\e^{itA/g_0}$ for each $t\in\R$ (see \cite[Sec.~4]{AP72}). This
occurs for instance when $A$ and $H$ are respectively the generator of dilations and
the Laplacian in $\R^d$ (see \cite[Sec.~1.2]{ABG96}), or when $A$ and $H$ are
respectively the generator of the geodesic flow and the generator of the horocycle
flow on a compact surface of constant negative curvature (see \cite[Sec.~4]{AP72}).
\end{Example}

\begin{Example}[Adjacency operators on admissible graphs]\label{ex_graphs}
Let $(X,\sim)$ be a graph, with symmetric relation $\sim$, of finite degree and with
no multiple edges or loops, and let $\H:=\ell^{\;\!2}(X)$ be the corresponding Hilbert
space. Then, the adjacency operator on $(X,\sim)$ defined by
$$
(H\varphi)(x):=\sum_{y\sim x}\varphi(y),\quad\varphi\in\H,~x\in X,
$$
is a bounded self-adjoint operator in $\H$. A directed graph $(X,<)$ subjacent to
$(X,\sim)$ is the graph $(X,\sim)$ together with a relation $<$ on $X$ such that, for
each $x,y\in X$, $x\sim y$ is equivalent to $x<y$ or $y<x$, and one cannot have both
$y<x$ and $x<y$. For each $x\in X$, we define $N^-(x):=\{y\in X\mid x<y\}$ and
$N^+(x):=\{y\in X\mid y<x\}$, and note that $N(x):=N^-(x)\cup N^+(x)$ is the set of
neighbours of $x$. When using drawings, we draw an arrow $x\leftarrow y$ if $x<y$. The
following definition is taken from \cite{MRT07}\;\!: A directed graph $(X,<)$ is
called admissible if
\begin{enumerate}
\item[(i)] each closed path in $X$ has the same number of positively and negatively
oriented edges
\item[(ii)] for each $x,y\in X$, one has\,
$\#\{N^-(x)\cap N^-(y)\}=\#\{N^+(x)\cap N^+(y)\}$.
\end{enumerate}

\begin{figure}[htbp]
\begin{center}
\includegraphics[width=320pt]{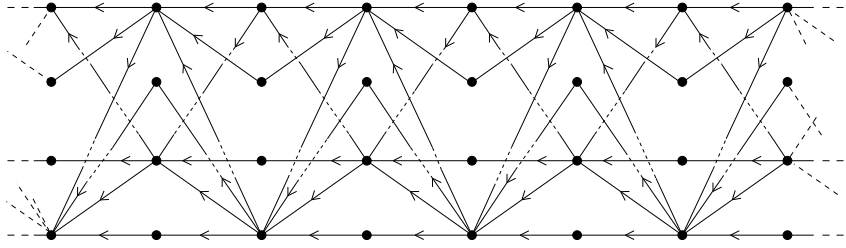}
\vspace{-10pt}
\caption{Example of an admissible graph}
\vspace{-5pt}
\end{center}
\end{figure}

Commutator calculations show that if $H$ is the adjacency operator on an admissible
graph $(X,<)$, then there exist a self-adjoint operator $A$ and a bounded self-adjoint
operator $K$ in $\H$ such that $H\in C^1(A)$ with $[iH,A]=K^2$, $[K,H]=0$, and
$K^2\in C^1(A)$ (see \cite{MRT07}). Therefore, one has
$$
D:=\slim_{t\to\infty}\frac1t\int_0^t\d s\,\e^{isH}(H+i)^{-1}[iH,A](H-i)^{-1}\e^{-isH}
=K^2(H^2+1)^{-1},
$$
and
$$
D^{-1}\eta(D)\;\!\dom(A)
=\big(K^2(H^2+1)^{-1}\big)^{-1}\eta\big(K^2(H^2+1)^{-1}\big)\;\!\dom(A)
\subset\dom(A)
$$
for each $\eta\in C^\infty_{\rm c}(\R\setminus\{0\})$ due to
\cite[Prop.~5.1.5~\&~Thm.~6.2.5]{ABG96}. Since
$$
\ker(D)
=\ker\big(K^2(H^2+1)^{-1}\big)
=\ker(K^2)
=\ker(K),
$$
we thus conclude by Theorem \ref{thm_self}(b) and Remark \ref{Rem_self}(a) that
$H|_{\ker(K)^\perp}$ has purely absolutely continuous spectrum.

This result was first proved in \cite[Thm.~1.1]{MRT07}. Even so, it is of interest
because the proof here is much simpler and shorter than in \cite{MRT07}. Similar
simplifications can also be obtained for the spectral analysis of  convolution
operators on locally compact groups as presented in \cite{MT07}. Details are left to
the reader.
\end{Example}

\begin{Example}[Time changes of horocycle flows]\label{ex_horocycle}
Let $\Sigma$ be a compact Riemann surface of genus $\ge2$ and let $M:=T^1\Sigma$ be
the unit tangent bundle of $\Sigma$. The compact $3$-manifold $M$ carries a
probability measure $\mu_\Omega$, induced by a canonical volume form $\Omega$, which
is preserved by two distinguished one-parameter groups of diffeomorphisms\;\!: the
horocycle flow $F_1\equiv\{F_{1,t}\}_{t\in\R}$ and the geodesic flow
$F_2\equiv\{F_{2,t}\}_{t\in\R}$. Each flow admits a self-adjoint generator $H_j$ in
$\ltwo(M,\mu_\Omega)$ which is essentially self-adjoint on $C^\infty(M)$ and given by
$$
H_j\hspace{1pt}\varphi:=-i\;\!\L_{X_j}\varphi,\quad\varphi\in C^\infty(M),
$$
with $X_j$ the divergence-free vector field associated to $F_j$ and $\L_{X_j}$ the
corresponding Lie derivative. Moreover, $H_1$ is of class $C^1(H_2)$ with
$[iH_1,H_2]=H_1$.

Now, consider a $C^1$ vector field proportional to $X_1$, namely, a vector field
$fX_1$ with $f\in C^1\big(M;(0,\infty)\big)$. The vector field $fX_1$ admits a
complete flow $\widetilde F_1\equiv\{\widetilde F_{1,t}\}_{t\in\R}$ with generator
$H:=fH_1$ essentially self-adjoint on $C^1(M)\subset\H:=\ltwo(M,f^{-1}\mu_\Omega)$.
Furthermore, Lemmas 3.1 and 3.2 of \cite{Tie13} imply the following\hspace{1pt}: The
operator $A:=f^{1/2}H_2f^{-1/2}$ is essentially self-adjoint on $C^1(M)\subset\H$, and
$(H-i)^{-1}\in C^1(A)$ with
$$
(H+i)^{-1}[iH,A](H-i)^{-1}=(H+i)^{-1}\big(H\;\!\xi+\xi\;\!H\big)(H-i)^{-1}
\quad\hbox{and}\quad
\xi:=\frac12-\frac12\;\!f^{-1}\L_{X_2}(f).
$$
It follows for $t>0$ that
$$
D_t
:=\frac1t\int_0^t\d s\,\e^{isH}(H+i)^{-1}[iH,A](H-i)^{-1}\e^{-isH}
=(H+i)^{-1}\big(H\;\!\xi_t+\xi_t\;\!H\big)(H-i)^{-1}
$$
with
$$
\xi_t
:=\frac1t\int_0^t\d s\,\e^{isH}\xi\e^{-isH}
=\frac1t\int_0^t\d s\,\big(\xi\circ\widetilde F_{1,-s}\big).
$$
Since the flow $F_1$ is uniquely ergodic with respect to the measure $\mu_\Omega$
\cite[p.~114]{Fur73}, the time changed flow $\widetilde F_1$ is uniquely ergodic with
respect to the measure
$\widetilde\mu_\Omega:=\frac{f^{-1}\mu_\Omega}{\int_Mf^{-1}\d\mu_\Omega}$
\cite[Prop.~3]{Hum74}. Therefore,
$$
\lim_{t\to\infty}\xi_t
=\frac12-\frac12\int_M\d\widetilde\mu_\Omega\;\!f^{-1}\L_{X_2}(f)
=\frac12+\frac1{2\int_Mf^{-1}\d\mu_\Omega}\int_M\d\mu_\Omega\;\!\L_{X_2}\big(f^{-1}\big)
=\frac12
$$
uniformly on $M$, and
$$
D:=\lim_{t\to\infty}D_t\textstyle=(H+i)^{-1}
\left(H\cdot\frac12+\frac12\cdot H\right)(H-i)^{-1}
=H(H^2+1)^{-1}
\quad\hbox{in}\quad\B(\H).
$$
So, Corollary \ref{first_cor_self} applies with $g(H)=H(H^2+1)^{-1}$ and
$\Lambda=\R\setminus\{0\}$, and thus
$\lim_{t\to\infty}\big\langle\varphi,\e^{-itH}\psi\big\rangle=0$ for each
$\varphi\in\H$ and $\psi\in\ker(H)^\perp$. This shows that the time changed horocycle
flow $\widetilde F_1$ associated to the vector field $fX_1$ is strongly mixing.

Thus, we have obtained a very simple proof that all time changes of horocycle flows of
class $C^1$ on compact surfaces of constant negative curvature are strongly mixing. We
refer to the work of A. G. Kushnirenko \cite{Kus74} for the same result under a
smallness assumption on $\L_{X_2}(f)$, and to the work of B. Marcus \cite{Mar77} for
the mixing properties of a general class of parametrisations of horocycle flows on
compact surfaces of negative curvature. We also refer to the works of G. Forni and C.
Ulcigrai \cite{FU12} and of the author \cite{Tie13,Tie12} for recent results on the
asolutely continuous spectrum of time changes of horocycle flows in the case of
constant curvature.
\end{Example}


\def\polhk#1{\setbox0=\hbox{#1}{\ooalign{\hidewidth
  \lower1.5ex\hbox{`}\hidewidth\crcr\unhbox0}}}
  \def\polhk#1{\setbox0=\hbox{#1}{\ooalign{\hidewidth
  \lower1.5ex\hbox{`}\hidewidth\crcr\unhbox0}}}
  \def\polhk#1{\setbox0=\hbox{#1}{\ooalign{\hidewidth
  \lower1.5ex\hbox{`}\hidewidth\crcr\unhbox0}}} \def\cprime{$'$}
  \def\cprime{$'$} \def\polhk#1{\setbox0=\hbox{#1}{\ooalign{\hidewidth
  \lower1.5ex\hbox{`}\hidewidth\crcr\unhbox0}}}
  \def\polhk#1{\setbox0=\hbox{#1}{\ooalign{\hidewidth
  \lower1.5ex\hbox{`}\hidewidth\crcr\unhbox0}}} \def\cprime{$'$}
  \def\cprime{$'$}


\end{document}